\documentclass[11pt]{amsart}
\usepackage{amsfonts}
\usepackage{amssymb}

\def\a{\alpha}		 \def\b{\beta}		  
\def\d{\delta}		 \def\De{{\Delta}}		 \def\e{\varepsilon}
\def\la{\lambda}	 		\def\O{\Omega}
\def\s{\sigma}		 \def\t{\theta}		  \def\f{\varphi}
		 \def\r{\rho}		  \def\z{\zeta}

	 \def \cb{{\mathcal B}}
	 \def \cd{{\mathcal D}}
	 
	 \def \ch{{\mathcal H}}

\def\C{{\mathbb C}}	 \def\D{{\mathbb D}}
	 
	 \def\N{{\mathbb N}}
	 
\def\T{{\mathbb T}}	 

\def\({\left(}		 \def\){\right)}

\newtheorem{prop}{\sc Proposition}
\newtheorem{lem}{\sc Lemma}
\newtheorem{thm}{\sc Theorem}
\newtheorem{cor}{\sc Corollary}
\newtheorem{ex}{\sc Example}

\begin{document}
\title[Composition operators from the Bloch space]
{Compact and weakly compact composition operators from the Bloch
space into M\"obius invariant spaces}

\author[M.D. Contreras]{Manuel D. Contreras}
\address{Departamento de Matem\'atica Aplicada II and IMUS, Universidad de Sevilla, Camino de los Descubrimientos s/n, 41092 Seville, Spain}
\email{contreras@us.es}
 \urladdr{http://personal.us.es/contreras/}

\author[S. D\'{\i}az-Madrigal]{Santiago D\'{\i}az-Madrigal}
\address{Departamento de Matem\'atica Aplicada II and IMUS, Universidad de Sevilla, Camino de los Descubrimientos s/n, 41092 Seville, Spain}
\email{madrigal@us.es}

\author[D. Vukoti\'c]{Dragan Vukoti\'c}
\address{Departamento de Matem\'aticas, M\'odulo 17, Facultad de Ciencias, Universidad Aut\'onoma de Madrid, 28049 Madrid, Spain} \email{dragan.vukotic@uam.es}
 \urladdr{http://www.uam.es/dragan.vukotic}
\dedicatory{Dedicated to Fernando P\'erez-Gonz\'alez on the occasion of his 60th birthday}
\date{Revised version: January 24, 2014}
\thanks{The first two authors were supported by MTM2012-37436-C02-01  from MINECO, Spain and FEDER (European Regional Development Fund) and the third author by MTM2012-37436-C02-02 from MINECO, Spain.}
\begin{abstract}
We obtain exhaustive results and treat in a unified way the question of boundedness, compactness, and weak compactness of composition operators from the Bloch space into any space from a large family of conformally invariant spaces that includes the classical spaces like $BMOA$, $Q_\a$, and analytic Besov spaces $B^p$. In particular, by combining techniques from both complex and functional analysis, we prove that in this setting weak compactness is equivalent to compactness. For the operators into the corresponding ``small'' spaces we also characterize the boundedness and show that it is equivalent to compactness.
\end{abstract}
\maketitle
 \tableofcontents
\section{Introduction}
\subsection{Boundedness and compactness of composition operators}
 \label{subsec-overview}
By a \textit{self-map\/} of the unit disk $\D$ we will mean an  analytic function $\f$ from the unit disk $\D$ into itself. Every  self-map of $\D$ induces the \emph{composition operator\/} $C_\f$ with \emph{symbol\/} $\f$ by the formula $C_\f(f)= f\circ\f$ on the set of all analytic functions in $\D$ but it is often of interest to consider $C_\f$ as an operator between Banach spaces of analytic functions. For several classical spaces of analytic functions such as a Hardy space $H^p$, a Bergman space $A^p$, or the Bloch space $\cb$, any symbol $\f$ gives rise to a bounded operator $C_\f$ from the space into itself. However, this is not the case for the Dirichlet space or for more general analytic Besov spaces  $B^p$, so the question of deciding which $\f$ induce a bounded operator $C_\f$ is of interest. The situation becomes more complicated if we consider composition operator acting between two different spaces.
\par
A related problem is to characterize all compact or weakly compact operators $C_\f$ between two given spaces in terms of the symbols $\f$. Criteria for compactness of $C_\f$ when acting on Hardy and Bergman spaces (due to J.H. Shapiro and B. MacCluer) are now already considered a classical knowledge; see \cite{CMC}, \cite{Sh}. For compact operators acting on $\cb$ and on the little Bloch space $\cb_0$ we refer the reader to \cite{MM}, \cite{MR}, and \cite{CHD}. Related results regarding composition operators from $\cb$ into the Dirichlet space $\cd$ or the more general analytic Besov spaces $B^p$ can be found in \cite{M1}, \cite{M2}, and \cite{Z}. Composition operators from $\cb$ into Hardy spaces were treated in \cite{PGX} while those from the Bloch space into the conformally invariant subspaces $BMOA$ and $VMOA$ of Hardy spaces and other spaces were studied in \cite{T} and \cite{MT}. For composition operators from $\cb$ into $Q_p$-type spaces we refer the reader to \cite{SZ}. Weak compactness of composition operators on vector-valued versions of classical spaces of analytic functions have been considered in \cite{BDL}, for example.
\par
Obviously, there are quite a few papers on the subject but it turns out that many similar setups are treated in an isolated way and many proofs are essentially repeated while it looks desirable to show the ``bigger picture''. One purpose of this paper is precisely to treat such questions globally, for those $C_\f$ that map the Bloch space into other spaces. We would like to underline that our work also provides new results in the case when the target space is one of the many rather classical Banach spaces.
\subsection{Conformally invariant spaces and weak compactness}
 \label{subsec-conf-weak}
It is traditional in the literature to consider the spaces $X$ which
are M\"obius invariant, \textit{i.e.}, those whose seminorm $s$ has the following property: $s(f\circ\s)\le C\,s(f)$, $f\in X$, for some fixed constant $C$ and all disk automorphisms $\s$. These spaces were given a systematic treatment in the seminal paper \cite{AFP} which was also pioneering in the theory of composition operators acting on them. This family of spaces includes the Bloch space $\cb$, the little Bloch space $\cb_0$, and analytic Besov spaces denoted $B^p$. We also mention the important spaces $BMOA$ (a variant of the classical John-Nirenberg space $BMO$) and $VMOA$ (introduced by Sarason; see \cite{G}), both M\"obius-invariant subspaces of the Hardy space $H^2$. The classical Hardy and Bergman spaces, however, do not satisfy the requirements for belonging to this family.
\par
The question whether the weak compactness of a composition operator  acting between two conformally invariant spaces of analytic functions  is actually equivalent to its compactness has generated considerable interest among the experts. For the composition operators on $BMOA$ or $VMOA$ or between these spaces, this question was posed (in its different versions) by Bourdon, Cima, and Matheson \cite{BCM}, \cite{CM}, by Laitila \cite{L}, and also by Tjani. An affirmative answer has been given recently by Laitila, Nieminen, Saksman, and Tylli \cite{LNST}, where the authors used some functional analysis tools such as the Bessaga-Pe\l czy\'nski selection principle.
\par
It is important to notice that there exist weakly compact composition operators acting on other function spaces which are not compact. An example of such $C_\f$, induced by a lense map $\f$, was given in \cite{LLQR}.
\subsection{Aims of this paper}
 \label{subsec-aims}
\par
The idea of considering the largest conformally invariant subspace of a given Banach space of analytic functions has already been considered in the literature for some time. Two relevant sources are \cite{A} and \cite{X1}. Significant motivation for our work comes from the approach adopted by Aleman and Simbotin (Persson); see \cite{AS} or \cite{Si}. \par
In the present paper we consider three fairly large families of spaces of analytic functions: the spaces $D^p_\mu$ defined in terms of integrability of the derivative of a function with respect to a certain Borel measure $\mu$, their conformally invariant subspaces $M(D^p_\mu)$, and the small subspaces $M_0(D^p_\mu)$. We defer until Subsection~\ref{subsec-conf-inv-sp} their precise definitions which coincide with those given in \cite{AS} or \cite{Si}. These families include various types of well-known spaces:
\par
1) the Hardy space $H^2$ and all weighted Bergman and Dirichlet-type spaces,
\par
2) their M\"obius invariant subspaces such as $BMOA$, $\cb$, analytic Besov spaces, and $Q_p$ spaces, and
\par
3) the small subspaces of the above spaces such as $VMOA$, $\cb_0$ or $Q_{p,0}$.
\par
It should also be remarked that families of ``large'' and ``small'' spaces defined by means of oscillation and density of polynomials in them (which is also discussed here) were considered in Perfekt's recent paper \cite{P}.
\par\smallskip
The main purpose of this paper is to present a unified approach to characterizing all bounded, compact, and weakly compact composition operators from $\cb$ into any of the spaces belonging to the family mentioned above. Our principal result shows that every weakly compact composition operator from $\cb$ into any space $M(D^p_\mu)$ is actually compact. We also generalize a number of existing but scattered results and add some new results. For instance, we characterize the compact and weakly compact operators from the Bloch space into the space $BMOA$, a question that has been left open in the literature. We do this by using a combination of complex analysis arguments and Banach space techniques that seem new in this context.
\par
Part of the motivation for our approach to compactness comes from Xiao's treatment \cite[Section~2.2]{X2}. However, to the best of our knowledge, our approach to weak compactness  appears to be different from the methods used in the existing literature.
\subsection{Main results}
 \label{subsec-results}
First of all, we characterize completely and in terms of the hyperbolic derivative of the symbol $\f$ all bounded and compact composition operators $C_\f$ from the Bloch space $\cb$ into any of the general spaces $D^p_\mu$, $M(D^p_\mu)$, and $M_0(D^p_\mu)$ considered in the paper. It turns out that whenever $C_\f\,\colon\,\cb\to D^p_\mu$ or $C_\f\,\colon\,\cb\to M_0(D^p_\mu)$, the compactness of $C_\f$ follows ``for free'' (after some work).
\par
Our Theorem~\ref{thm-cpt-inv} (see Section~\ref{sec-cpt-ops}) describes the compact composition operators from $\cb$ into the invariant space $M(D^p_\mu)$ and shows that, in this case, weak compactness is equivalent to compactness. The proof is based on a theorem of Banach-Saks type from functional analysis and techniques from function spaces. The result is accompanied by appropriate examples.
\par
Another relevant point (see Section~\ref{sec-little-sp}, Theorem~\ref{thm-dens-polyn}) is a rigorous and detailed proof that, for all natural radial measures of certain type, the polynomials are dense in the small subspace $M_0(D^p_\mu)$ of the conformally invariant space $M(D^p_\mu)$, in analogy with the classical cases. This provides a wide range of examples where the separability hypothesis of our Theorem~\ref{thm-small} is satisfied. This last result characterizes the bounded and compact composition operators from the Bloch space $\cb$ into the small spaces $M_0(D^p_\mu)$.
\subsection*{Acknowledgments}
 \label{subsec-acknowledge}
The authors are grateful to the anonymous referee for reading carefully the first version of the manuscript and pointing out several misprints, as well as for suggesting some additional references and making some useful remarks that helped improve the exposition.
\par\medskip
\section{A family of conformally invariant spaces}
 \label{fam-conf-inv}

\subsection{Notation}
 \label{subsec-notation}
In what follows, $\D$ will denote the unit disk in the complex plane: $\D = \{z\in\C\,\colon\,|z|<1\}$ and $dA$ will denote the normalized Lebesgue area measure on $\D$:
$$
 dA(z)=\frac1{\pi}\,dx\,dy=\frac1{\pi}\,r\,dr\,d\t\,, \quad z=x+yi=r e^{i\t}\,.
$$
By a \emph{disk automorphism}, we will mean  a one-to-one analytic mapping of $\D$ onto itself. The set of all such maps, $Aut(\D)$, is a transitive group under composition. As is well known, every $\s\in Aut(\D)$ has the form
\begin{equation}
 \s(z)= \la \frac{a-z}{1-\overline{a} z}\,, \quad |\la|=1\,, \quad |a|<1\,.
 \label{autom}
\end{equation}
An important property of disk automorphisms is that they yield equality in
the Schwarz-Pick lemma:
\begin{equation}
 (1-|z|^2) |\s^\prime (z)| = 1-|\s(z)|^2\,, \quad z\in\D\,.
 \label{ident-aut}
\end{equation}
\par
Throughout the paper, we shall always consider a positive Borel
measure $\mu$ on $\D$. A typical example is
$$
 d\mu(z)=(1-|z|^2)^\a dA(z)\,,
$$
a measure which is finite if and only if $-1<\a<\infty$. Another example is
$$
 d\mu(z)=\log^\a \frac1{|z|} dA(z)\,.
$$
Note that for $z$ near the unit circle the function $\log^\a \frac1{|z|}$ behaves asymptotically like $(1-|z|^2)^\a$. In principle, our measures are not assumed to be of the form $h(|z|)\,dA(z)$, where $h$ is some integrable positive function on $[0,1)$ like in the above examples. However, the result in the last section will mostly be displayed for measures that satisfy this assumption.
\par\medskip
\subsection{The basic conformally invariant spaces}
 \label{subsec-conf-inv-sp}
Throughout the paper, we will use $\ch(\D)$ to denote the set of all functions analytic in $\D$. A function $f\in\ch(\D)$ is said to belong to the \textit{Bloch space\/} $\cb$ if its \textit{invariant derivative\/}: $(1-|z|^2) |f^\prime(z)|$ is bounded in $\D$. The name comes from the fact that this quantity does not change under a composition with any $\s\in Aut(\D)$ in view of our formula \eqref{ident-aut}. The Bloch space becomes a Banach space when equipped with the norm
$$
 \|f\|_\cb = |f(0)|+ \sup_{z\in\D} (1-|z|^2) |f^\prime(z)| \,.
$$
Every function in $\cb$ satisfies the standard growth condition:
\begin{equation}
 |f(z)| \le \(1+\frac12 \log\frac{1+|z|}{1-|z|}\)\,\|f\|_\cb \,, \quad z\in\D \,.
 \label{bl-growth}
\end{equation}
\par
Given a positive Borel measure $\mu$ on $\D$ and $p\in [1,\infty)$, we can define the \textit{weighted Dirichlet-type spaces\/} $D^p_\mu$ in the usual way:
$$
 D^p_\mu = \left\{f\in\ch(\D)\,\colon\,\left\Vert f\right
 \Vert _{D^p_\mu}^{p}:=\left\vert f(0)\right\vert ^{p}+\int_{\D}
 \left\vert f^{\prime }\right\vert^{p}\,d\mu<\infty\right\} \,.
$$
Consider the point evaluation functionals $\phi_\z$, defined by $\phi_\z (f) = f(\z)$, for $\z\in\D$. It is natural to require the following axioms to hold:
\begin{itemize}
\item
 $D^p_\mu$ is a Banach space;
\item
 The point-evaluation functional $\phi_\zeta$ is bounded on $D^p_\mu$ for each $\z\in\D$.
\end{itemize}
In view of the uniform boundedness principle, these two requirements can be summarized in one single axiom:
\begin{itemize}
\item
\emph{The point-evaluation functionals are uniformly bounded on $D^p_\mu$ on compact subsets of $\,\D$.}
\end{itemize}
\par
Following the notation used, for example, in \cite{AS}, we define the \textit{M\"obius invariant subspace\/} $M(D^p_\mu)$ as the space of
all functions $f$ in $\ch(\D)$ such that
$$
 \left\Vert f \right\Vert _{M(D^p_\mu)}^{p}:=\left\vert f(0)\right\vert^{p} +  \sup_{\s\in Aut(\D)}\int _{\D}\left\vert (f\circ\s)^{\prime} \right\vert ^{p}\,d\mu <\infty\,.
$$
We also define the corresponding \textit{little invariant subspaces\/}:  $$
 M_0(D^p_\mu) = \left\{ f\in M(D^p_\mu)\,\colon\,\lim_{|\s(0)|
 \to 1,\,\s\in Aut(\D)} \,\int_{\D} \left\vert (f\circ\s)^{\prime}\right\vert^{p}\,d\mu = 0 \right\} \,.
$$
A few remarks are in order:
\begin{itemize}
\item
It is routine to verify that $s(f)=\sup_{\s\in Aut(\D)} \(\int_{\D}\left\vert (f\circ\s)^{\prime}\right\vert ^{p}\,d\mu\)^{1/p}$ defines a seminorm on $M(D^p_\mu)$ and $\left\Vert \cdot \right\Vert_{M(D^p_\mu)}$ has all the properties of a norm.
\item
Since $\{\tau\circ \s\,\colon\,\s\in Aut(\D)\}=Aut(\D)$ holds for any fixed $\tau\in Aut(\D)$, it follows that the $s(f\circ\tau)= s(f)$. In other words, this seminorm is conformally invariant.
\item
Since the identity map of $\D$ is trivially a disk automorphism, it is immediate that $M(D^p_\mu)\subset D^p_\mu$. It actually follows  from our previous comment that $M(D^p_\mu)$ is the largest conformally invariant subspace of  $D^p_\mu$.
\item
Note that we actually require that $f\in M(D^p_\mu)$ in the definition of $M_0(D^p_\mu)$ since it is not obvious, even for somewhat special measures $\mu$, that the assumption
$$
 \lim_{|\s(0)|\to 1,\,\s\in Aut(\D)} \,\int_{\D} \left\vert (f\circ\s)^{\prime}\right\vert ^{p}\,d\mu = 0
$$
implies that
$$\sup_{\s\in Aut(\D)}\int _\D \left\vert (f\circ\s)^{\prime}
 \right\vert ^{p}\,d\mu <\infty\,.
$$
\item
Assuming the uniform boundedness of point evaluations in $D^p_\mu$ on compact subsets of $\D$, by a standard normal families argument and Fatou's lemma one can deduce the completeness of $M(D^p_\mu)$. It is easily checked that $M_0(D^p_\mu)$ is a closed subspace of $M(D^p_\mu)$, so it is also complete.
\item
It is not difficult to see that each one of the spaces defined above contains sufficiently many functions for most ``reasonable'' measures $\mu$. For example, if $\mu$ is a finite measure then every function analytic in a disk larger than $\D$ and centered at the origin is readily seen to belong to $M(D^p_\mu)$. We shall discuss the membership and density of the polynomials in $M_0(D^p_\mu)$ in the final section of the paper.
\item
In several papers only the involutive automorphisms are considered: $\s_a(z)=(a-z)/(1-\overline{a}z)$, $a\in\D$, requiring that $|a|\to 1$ in the definitions of the special small spaces. Here we have opted for the full generality and for considering the entire automorphism group, which adds certain technical difficulties to some proofs.
\end{itemize}
\par\medskip
\subsection{Classical function spaces as examples}
 \label{subsec-conf-inv-ex}
An appropriate choice of $\mu$ in the above definitions of our spaces $D^p_\mu$, $M(D^p_\mu)$, and $M_0(D^p_\mu)$ yields a number of well-known spaces of analytic functions in the disk as special cases. Here is a list of some important examples.
\par\smallskip
(\textbf{A}) In view of the well-known Littlewood-Paley identity \cite[p.~51]{Sh}:
$$
 \|f\|_{H^2}^2 = |f(0)|^2 + 2 \int_\D |f^\prime(z)|^2 \log\frac{1}{|z|}\,dA(z)\,,
$$
the Hardy space $H^2$ can be seen as a $D^2_\mu$ space by choosing $d\mu(z)=\log \frac{1}{|z|^2}\,dA(z)$. Its conformally invariant subspace $M(D^2_\mu)$ is the well-known $BMOA$ space of analytic functions of bounded mean oscillation and the corresponding space  $M_0(D^2_\mu)=VMOA$, the space of functions of vanishing mean oscillation; see \cite{G} for more about these space. It should be remarked that in this case our definition involving all possible disk automorphisms coincides with the usual one that takes into account only the involutive automorphisms $\s_a$ mentioned above in view of rotation invariance of the measure $\mu$.
\par\smallskip
(\textbf{B}) The analytic Besov spaces $B^p$, $1<p<\infty$, are obtained as $D^p_\mu$ spaces by choosing $d\mu(z)=(p-1) (1-|z|^2)^{p-2} dA(z)$, $1<p<\infty$. See \cite{AFP} or \cite{Z1} for more about these spaces.
Note that, in this case, combining the simple change of variable $w=\s(z)$, $dA(w)=|\s^\prime (z)|^2 dA(z)$ with \eqref{ident-aut} shows that
$$
 \int _{\D}\left\vert (f\circ\s)^{\prime} (z) \right\vert ^{p}\,d\mu (z) = \int _{\D}\left\vert f^{\prime} (w) \right\vert ^{p}\,d\mu (w)
$$
so it is immediate that here $D^p_\mu=M(D^p_\mu)$ while the corresponding space $M_0(D^p_\mu)$ is trivial (consisting only of the constant functions).
\par\smallskip
(\textbf{C}) The Bergman spaces $A^p$, $1\le p<\infty$, can be obtained by taking $d\mu(z)=(1-|z|^2)^p dA(z)$. Well-known (but too lengthy to repeat here) arguments using the Cauchy integral formula and Minkowski's inequality in its integral form as in \cite[Chapter~5]{D} show that the norm in our definition is equivalent to the standard Bergman norm:
$$
 \left\vert f(0)\right\vert ^{p}+\int_{\D} |f^{\prime} (z)|^p \,(1-|z|^2)^p dA(z) \asymp \int_\D |f(z)|^p\,dA(z)
$$
(meaning that each of the two sides is bounded by a constant multiple of the other, this multiple being independent of $f$).
In this case it turns out that $M(D^p_\mu)=\cb$ and $M_0(D^p_\mu)=\cb_0$, the little Bloch space (the closure of polynomials in $\cb$), as was shown by Axler \cite{A}.
\par\smallskip
(\textbf{D}) The $Q_\a$ spaces, defined by Aulaskari, Xiao, and Zhao \cite{AXZ} and studied by other authors as well (see \cite{X2} for an extensive account), can be seen as $M(D^p_\mu)$ spaces by taking $p=2$, $d\mu(z)=\log^\a \frac1{|z|} dA(z)$, $0<\a<\infty$. An equivalent norm is obtained by choosing $d\mu = (1-|z|)^\a\, dA(z)$ instead. (Note that we will use the notation $Q_\a$ rather than the traditional $Q_p$ because here $p=2$ is fixed and the exponent $\a$ from the weight is the one that determines the space.) It is well known that $Q_\a$ coincides as a set with $\cb$ (but is, of course, endowed with a different norm) whenever $\a>1$ and with $BMOA$ when $\a=1$, while it is an entirely different space when $0<\a<1$. The corresponding small space $M_0(D^p_\mu)$ is the space usually denoted as $Q_{\a,0}$ and $Q_{1,0}=VMOA$.
\par\medskip
\section{Bounded composition operators}
 \label{sec-bded-ops}

\subsection{Preliminaries}
 \label{subsec-prelim}
The following lemma in the case $p=1$ has been proved explicitly by Ramey and Ullrich \cite{RU} although the argument can probably be traced back to Ahern and Rudin.
\begin{lem}
 \label{lem-ur}
Let $1\le p<\infty$. There exist two functions $f$ and $g$ in the
Bloch space $\cb$ and a positive constant $C$ such that
$$
 |f^\prime(z)|^p + |g^\prime (z)|^p \ge \frac{C}{(1-|z|^2)^p}
$$
for all $z$ in $\D$.
\end{lem}
The proof follows by \cite[Proposition~5.4, p.~601]{RU}. The key point is to select a partition of the disk into two sets of concentric annuli centered at the origin and two lacunary series, one of which takes on large enough values:
$$
 |f^\prime(z)| \ge \frac{A}{1-|z|^2}
$$
on the odd-numbered annuli and the other does the same on the
even-numbered annuli. This takes care of the case $p=1$; for arbitrary $p\ge 1$ the statement follows readily by the standard inequality $(a+b)^p\le 2^p (a^p+b^p)$, where $a$, $b\ge 0$.
\par\medskip
It will be convenient to use the following version of the \emph{hyperbolic derivative\/} of an analytic self-map $\f$ of $\D$:
$$
 \f^{\#}(z)=\dfrac{\left\vert \f^{\prime}(z)\right\vert}{1-
 \left\vert \f(z)\right\vert ^{2}}\,.
$$
It should be noted that there is another related quantity also called hyperbolic derivative but only the above expression will be useful for our purpose.

\subsection{Characterizations of boundedness}
 \label{subsec-bounded}
\par
We state two basic facts which characterize the bounded composition operators from $\cb$ into $D^p_\mu$ and into $M(D^p_\mu)$ respectively. The proofs of such facts are relatively straightforward and have by now become standard in the literature. We record them here only for the sake of completeness.
\par\smallskip
\begin{prop}
 \label{prop-bded-basic}
The following statements are equivalent:
\begin{itemize}
 \item[(a)] $C_\f\colon \cb\to D^p_\mu$;
 \item[(b)] $C_\f$ is a bounded operator from $\cb$ into
	 $D^p_\mu$;
 \item[(c)] $\int_\D |\f^\#|^p\,d\mu<\infty$.
\end{itemize}
\end{prop}
\begin{proof}
It suffices to verify the following short chain of implications: (a)\,$\Rightarrow$\,(c)\,$\Rightarrow$\, (b)\,$\Rightarrow$\,(a).
\par\medskip
\fbox{(a)\,$\Rightarrow$\,(c).} Suppose that $f\circ\f\in D^p_\mu$ for each $f$ in $\cb$. Choose two functions $f$, $g \in\cb$ and the constant $C>0$ as in Lemma~\ref{lem-ur}. Evaluate them at $\f(z)$ to get that
$$
 \frac{C}{(1-|\f(z)|^2)^p} \le |f^\prime(\f(z))|^p + |g^\prime (\f(z))|^p\,, \quad z\in\D \,.
$$
This yields
\begin{eqnarray*}
 C \int_\D \frac{|\f^\prime|^p}{(1-|\f|^2)^p}\,d\mu &\le & \int_\D |f^\prime \circ \f|^p |\f^\prime|^p\,d\mu + \int_\D |g^\prime \circ \f|^p |\f^\prime|^p\,d\mu
\\
 &\le & \|f\circ\f\|_{D^p}^p + \|g\circ\f\|_{D^p}^p < \infty\,.
\end{eqnarray*}
This proves (c).
\par\medskip
\fbox{(c)\,$\Rightarrow$\,(b).} Suppose that $\int_\D |\f^\#|^p\,d\mu <\infty$. Let $f$ be an arbitrary function in $\cb$. Then
$$
 |f^\prime(\f(z))| (1-|\f(z)|^2) \le \|f\|_\cb
$$
for every $z\in\D$. This readily implies that
$$
 \int_\D |f^\prime \circ \f|^p |\f^\prime|^p\,d\mu \le \int_\D |\f^\#|^p\,d\mu \,\cdot\, \|f\|_\cb^p < \infty\,.
$$
Also, from the growth estimate \eqref{bl-growth} we obtain
$$
 |f(\f(0))|\le \(1+\frac12 \log\frac{1+|\f(0)|}{1-|\f(0)|}\)\, \|f\|_\cb\,.
$$
By summing up the last two inequalities, it follows that $C_\f$ is bounded as an operator from $\cb$ into $M(D^p_\mu)$.
\par\medskip
\fbox{(b)\,$\Rightarrow$\,(a)} is trivial.
\end{proof}
\par\medskip
\begin{prop}
 \label{prop-bded-inv}
The following statements are equivalent:
\begin{itemize}
 \item[(a)]
$C_\f\colon \cb\to M(D^p_\mu)$;
 \item[(b)]
$C_\f$ is a bounded operator from $\cb$ into $M(D^p_\mu)$;
 \item[(c)]
$\sup_{\s\in Aut(\D)} \int_\D |(\f\circ\s)^\#|^p
	\,d\mu<\infty$.
\end{itemize}
\end{prop}
\begin{proof}
The proof can be worked out along the same lines as that of Proposition~\ref{prop-bded-basic}, with the necessary modifications.
\end{proof}
\par\medskip
Since we are working in a fairly general context, it is important to make sure that we are not dealing with trivial situations by displaying  examples that work in a large number of cases. Here is a very simple example showing that very simple symbols may or may not yield bounded composition operators from $\cb$ to our spaces.
\par\medskip
\begin{ex}
 \label{ex-ident}
Let $1\le p<\infty$, $d\mu(z)=(1-|z|^2)^\a\,dA(z)$, and let $\f(z)\equiv z$. Then the following statements are equivalent:
\begin{itemize}
 \item[(a)]
$C_\f$ is bounded as an operator from $\cb$ into $D^p_\mu$.
 \item[(b)]
$C_\f$ is bounded as an operator from $\cb$ into $M(D^p_\mu)$.
 \item[(c)]
$p-\a<1$.
\end{itemize}
The case of $M(D^p_\mu)$ is slightly more involved, but still easy to check, in view of the identity \eqref{ident-aut}:
$$
 \int_\D |(\f\circ\s)^\#|^p\,d\mu = \int_\D |\s^\#(z)|^p\,d\mu(z) =
 \int_\D \frac{d\mu(z)}{(1-|z|^2)^p} = \int_\D \frac{dA(z)}{(1-|z|^2)^{p-\a}}\,.
$$
Trivial integration in polar coordinates shows that the last integral converges only for the range indicated in (c). (Note that this is really a statement about the containments $\cb\subset M(D^p_\mu)$ but is at the same time an example for composition operators.)
\end{ex}
\par\medskip
The case of bounded operators from $\cb$ into the little M\"obius invariant subspaces $M_0(D^p_\mu)$ will be considered in the last section together with the compactness question. This is done because the two turn out to be equivalent and the proof requires other results to be obtained first.
\subsection{Some remarks on the case $p=2$}
 \label{subsec-p=2}
The reader has probably noticed some differences in the formulation of our results and those by other authors pertaining to the cases like $BMOA$ or $Q_\a$. The reason for this is very simple: these spaces are obtained in the special case $p=2$ when some of our results above can be rewritten in a different language.
\par
To this end, denote by $N_\f$ the \textit{counting function\/} of $\f$:
$$
 N_\f(w)= |\{z\in\D\colon \f(z)=w\}|\,,
$$
understanding $0$, $1$, $2$,\ldots,$\infty$ as its possible values.
Let us also agree to write $A_h$ for the hyperbolic area of a subset of the disk:
$$
 A_h(S)=\int_S \frac{d A(z)}{(1-|z|^2)^2} \,.
$$
\begin{lem}
 \label{lem-hyp-equiv}
For arbitrary positive measure $\mu$, we have
\begin{equation}
 \frac{|\f^\prime|^2}{(1-|\f|^2)^2} = \De \log \frac1{1-|\f|^2} \,.
 \label{hyp-equiv-1}
\end{equation}
When $d\mu=dA$, we also have
\begin{equation}
 \int_\D \frac{|\f^\prime|^2}{(1-|\f|^2)^2} d A = \int_\D N_\f\,d A_h\,.
 \label{hyp-equiv-2}
\end{equation}
\end{lem}
\par\smallskip
Formula \eqref{hyp-equiv-1} is a simple consequence of the identity $\De(u\circ\f)=(\De u\circ\f)|\f^\prime|^2$ while \eqref{hyp-equiv-2} follows from the well-known formula for non-univalent change of variable (see \cite[Section~10.3]{Sh} or \cite[Theorem~2.32]{CMC}).
\par\smallskip
Taking into account the equivalent forms of writing $|\f^\#|$ from Lemma~\ref{lem-hyp-equiv}, it becomes obvious how the condition (c) in Proposition~\ref{prop-bded-basic} and Proposition~\ref{prop-bded-inv} can be rewritten. For example, in two special cases we could state our Proposition~\ref{prop-bded-basic} or Proposition~\ref{prop-bded-inv} as follows:
\begin{itemize}
 \item
\textit{For arbitrary $\mu$, the composition operator $C_\f$ is bounded from $\cb$ into $M(D^2_\mu)$ if and only if}
$$
 \sup_{\s\in Aut(\D)} \int_\D \De\log\frac1{1-|\f\circ\s|^2}\,d\mu <\infty \,.
$$
\item
\textit{$C_\f$ is bounded from $\cb$ into $D^2_A$ ($A$ being the area measure) if and only if
$$
 \sup_{\s\in Aut(\D)}  \int_\D N_{\f\circ\s}\,d A_h = \int_\D N_{\f}\,d A_h < \infty\,,
$$
in view of conformal invariance.}
\end{itemize}
\par\medskip
\section{Compact and weakly compact composition operators}
 \label{sec-cpt-ops}
\par
Recall that a bounded linear operator between two Banach spaces is
said to be \emph{compact} if it takes bounded sets into sets whose
closure is compact; equivalently, if for every bounded sequence in
the space the sequence of images has a convergent subsequence in the
norm topology. A bounded operator is \emph{weakly compact} if it
takes bounded sets into sets whose closure is weakly compact;
equivalently, if for every bounded sequence in the space the sequence
of images has a subsequence that converges in the weak topology.
Compactness obviously implies weak compactness.
\par
We will now show that every composition operator from the Bloch space
$\cb$ into any of our spaces $M(D^p_\mu)$ is compact if and only if it
is weakly compact and will also give a characterization of this
property in terms of the symbol $\f$ which unifies all previously
obtained results for concrete spaces. In the special case of composition operators from $\cb$ to $Q_\a$, the equivalence of (a) and (c) in Theorem~\ref{thm-cpt-inv} below has been proved before by Smith and Zhao \cite{SZ}; see also \cite{X1} or \cite{Sm}. However, weak compactness was not considered in these works.
\par
The main novelty of our approach consists of the use of certain
techniques usually employed by the experts in Banach space theory,
the main one being a version of the Banach-Saks theorem. We formulate below the statement needed as a lemma but remark that it proof relies on some rather non-trivial results. It should be observed that the lemma is no longer true (even for composition operators) if we only assume boundedness of the operator.
\begin{lem}
 \label{lem-b-s}
Suppose that $T$ is a weakly compact operator from $\cb$ into an arbitrary Banach space $Y$. Then every bounded sequence $(f_{n})_n$ in $\cb$ has a subsequence $(f_{n_{k}})_k$ such that the arithmetic means of the images $Tf_{n_k}$ converge to some element in the norm of \/$Y$.
\end{lem}
\begin{proof}
Recall that the Bloch space is isomorphic to the space of all bounded complex sequences $l^{\infty}$; see Lusky's paper \cite{Lu}. On the other hand, $l^{\infty}$ is a unital commutative $C^{\ast}$-algebra endowed with the usual operations of coordinatewise multiplication and conjugation. The Gelfand-Naimark theorem (see \cite[Thm.~11.18]{Ru} where the author uses the term $B^\ast$-algebra instead) now implies that $l^{\infty}$ is isomorphic to a space of continuous functions on its maximal ideals (which is a compact Hausdorff space by \cite[Thm.~11.9]{Ru}). Thus, we are allowed to apply a Banach-Saks type theorem proved in 1979 by Diestel - Seifert and Niculescu (see \cite[p.~320]{DJT}) which establishes that any weakly compact linear operator from a space of continuous functions on a compact Hausdorff space into an arbitrary Banach space has the Banach-Saks property.
\par
Alternatively, we could have deduced the statement from a more general result of Jarchow referring directly to the $C^{\ast}$-algebras (see also \cite[p.~320]{DJT}).
\end{proof}
\par\medskip
Note that the measure $\mu$ in the theorem below is not required to be of any special form. In particular, it need not be finite.
\par
For the sake of brevity, throughout the rest of the paper we will write simply  $\{|\f\circ\s|>r\}$ to denote the set $\{z\in\D\colon |(\f\circ\s) (z)|>r\}$.
\par\medskip
\begin{thm}
 \label{thm-cpt-inv}
Let $1\leq p<\infty$. Suppose that $C_\f$ is a bounded operator from
$\cb$ to $M(D^p_\mu)$. Then the following statements are equivalent:
\begin{itemize}
 \item[(a)]
$C_\f$ is a compact operator from $\cb$ to $M(D^p_\mu)$.
 \item[(b)]
$C_\f$ is a weakly compact operator from $\cb$ to $M(D^p_\mu)$.
 \item[(c)]
$\lim_{r\to 1} \sup_{\s\in Aut(\D)} \int_{\{|\f\circ\s|>r\}} |(\f\circ\s)^\#|^p\,d\mu=0$.
\end{itemize}
\end{thm}
\begin{proof}
We proceed to prove the statement by proving the implications (c)\,$\Rightarrow$ (a)\,$\Rightarrow$\,(b)\, $\Rightarrow$\,(c).
\par\medskip
\fbox{(c)\,$\Rightarrow$\,(a).} Suppose (c) holds. It is clear that if $(f_n)_n$ is a bounded sequence in the Bloch space that converges to zero uniformly on compact sets, then $\lim_{n\to\infty} f_n(\f (0))$ $=0$. Thus, let us concentrate on the second term that appears in the norm. Fix an arbitrary $\e>0$. Then there exists $r_0\in (0,1)$ such that
\begin{equation}
 \sup_{\s\in Aut(\D)} \int_{ \{|\f\circ\s|>r_0\} } |(\f\circ\s)^\#|^p \,d\mu < \frac{\e}{2^{p+1}}\,.
 \label{est-hyp}
\end{equation}
Let $(f_n)_n$ be an arbitrary sequence in $\cb$ with $\|f_n\|_\cb\le
1$ for all $n$. By a normal families argument, there exists a
subsequence which we denote by $(g_n)_n$ which converges uniformly on
compact sets to an analytic function $g$. From
$$
 |g_n^\prime (z)| \le \frac{\|g_n\|_\cb}{1-|z|^2} \le \frac{1}{1-|z|^2}\,, \quad z\in\D\,,
$$
it readily follows that $g$ enjoys the same estimate, hence $g\in\cb$ and $\|g\|_\cb\le 1$. Moreover, we also have that $|g_n^\prime (w) - g^\prime (w)| \le \frac{2}{1-|w|^2}$ for all $w$ in $\D$, hence
\begin{equation}
 |(g_n - g)^\prime \circ (\f\circ\s)|^p |(\f\circ\s)^\prime|^p \le 2^p |(\f\circ\s)^\#|^p
 \label{est-conv}
\end{equation}
holds throughout $\D$. In order to show that $C_\f (g_n)\to C_\f (g)$ in the $M(D^p_\mu)$ norm, we need to show that the integrals
$$
 \int_\D |(g_n - g)^\prime \circ (\f\circ\s)|^p |(\f\circ\s)^\prime|^p\,d \mu
$$
are uniformly small independently of $\s$ as $n\to\infty$. For this purpose, it is convenient to split the above integral into two (omitting the integrands below):
\begin{equation}
 \int_\D = \int_{\{|\f\circ\s|\le r_0\}} + \int_{\{|\f\circ\s|>r_0\}} \,.
 \label{int-split}
\end{equation}
By assumption, $C_\f$ is bounded from $\cb$ to $M(D^p_\mu)$ so in view of Proposition~\ref{prop-bded-inv}, we have
$$
 \sup_{\s\in Aut(\D)} \int_\D |(\f\circ\s)^\#|^p\,d\mu\le M
$$
for some fixed positive constant $M$. Given $\e>0$, by virtue of the uniform convergence: $(g_n - g)^\prime\to 0$ on the compact set $\{z\,\colon\,|z|\le r_0\}$, for large enough $n$ we have
\begin{multline*}
 \int_{\{|\f\circ\s|\le r_0\}} |(g_n - g)^\prime \circ (\f\circ\s)|^p |(\f\circ\s)^\prime|^p\,d \mu
\\
 \le \int_{\{|\f\circ\s|\le r_0\}} |(g_n - g)^\prime \circ (\f\circ\s)|^p \frac{|(\f\circ\s)^\prime|^p}{(1-|\f\circ\s|^2)^p}\,d \mu
\\
 < \frac{\e}{2 M} \int_{\{|\f\circ\s|\le r_0\}}  \frac{|(\f\circ\s)^\prime|^p}{(1-|\f\circ\s|^2)^p}\,d \mu
 \le \frac{\e}{2}\,.
\end{multline*}
Thus, for $n$ sufficiently large, the first integral in \eqref{int-split} can be made smaller than $<\e/2$.
\par
The second integral in \eqref{int-split} is smaller than $\e/2$ in view of the inequalities \eqref{est-hyp} and \eqref{est-conv}. This implies that
$$
 \|C_\f (g_n)- C_\f (g)\|_{M(D^p_\mu)} <\e
$$
for $n$ large enough, as asserted.
\par\medskip
\fbox{(a)\,$\Rightarrow$\,(b)} is obvious.
\par\medskip
\fbox{(b)\,$\Rightarrow$\,(c)} is the most intricate part of the proof. We follow the steps indicated below.
\par\smallskip
\underline{\it Step 1\/}: We first show that the weak compactness assumption on $C_\f$ implies
\begin{equation}
 \lim_{r\to 1} \sup_{\s\in Aut(\D)} \int_{ \{|\f\circ\s|>r\} } |(\f\circ\s)^\prime|^p\,d\mu = 0\,.
 \label{e1}
\end{equation}
This condition alone is apparently much weaker than (c). However, we
will eventually show that, together with the weak compactness of $C_\f$, it actually implies the desired condition (c).
\par
Thus, suppose that $C_\f$ is weakly compact from $\cb$ into $M(D^p_\mu)$ but \eqref{e1} does not hold. Then we can find a positive number $\d$, an increasing sequence $(\r_j)_j$ of numbers in $(0,1)$ such that $\lim_{n\to\infty} \r_j=1$, and a sequence of disk automorphisms $(\tau_j)_j$ such that
\begin{equation}
 \int_{\{ |\f\circ\tau_j|>\r_j\} } |(\f\circ\tau_j)^\prime|^p\,d\mu  \ge\d\,.
 \label{delta}
\end{equation}
For a positive integer $k,$ let us agree to write
$$
C_{k}=\frac{k(k+1)}2 \,.
$$
Next, choose recursively a subsequence $(m_n)_n$ of the integers in such a way that
$$
 m_0=1\,, \quad m_n > C_{m_{n-1}} + n \,.
$$
Once the sequence $(m_n)_n$ has been fixed, let us choose the subsequence $(r_n)$ of the sequence $(\r_j)_j$ so that
$$
 m_n r_n^{m_n-1} > C_{m_{n-1}} + n \,, \quad n\ge 1\,.
$$
This is possible since $\lim_{r\to 1} m_n r^{m_n-1}=m_n$. Note that then
\begin{equation}
 m_n r_n^{m_n-1} > C_{m_{n-1}} + n \ge m_{n-1}+n > m_{n-1}\,.
 \label{sequence}
\end{equation}
Also, let us choose the subsequence $(\s_n)_n$ of $(\tau_j)_j$ with the same indices as those of $(r_n)_n$ with respect to $(\r_j)_j$.
\par
By applying Lemma~\ref{lem-b-s} to our weakly compact operator $C_{\f}\,\colon\,\mathcal{B}\to M(D^p_\mu)$ and observing that the sequence $(z^{m_{n}})_{n}$ is bounded in the Bloch space, we conclude that there exists a subsequence $(m_{n_k})_k$ of $(m_n)_n$ for which the arithmetic means
$$
 \frac{1}{N} \sum_{k=1}^N \f^{m_{n_{k}}}
$$
converge in the norm of $M(D^p_\mu)$. They actually must tend to zero since they converge to zero uniformly on compact sets. Hence,
$$
 \left\| \frac{1}{N} \sum_{k=1}^N \f^{m_{n_{k}}} \right\|^p_{M(D^p_\mu)} \ge \sup_{\s}\int_{\D}\left\vert  \frac{1}{N} \sum_{k=1}^{N} m_{n_{k}} (\f^{m_{n_{k}}-1}\circ\s)\right\vert ^p |(\f\circ\s)^\prime|^p \,d\mu \,,
$$
hence
$$
 \lim_{N\to\infty} \sup_{\s}\int_{\D}\left\vert  \frac{1}{N} \sum_{k=1}^{N} m_{n_{k}} (\f^{m_{n_{k}}-1}\circ\s)\right\vert ^p |(\f\circ\s)^\prime|^p \,d\mu = 0\,.
$$
Let $\e>0$ be fixed. There exists a positive integer $N>1$ such that
$$
 \sup_{\s}\int_{\D}\left\vert  \frac{1}{N} \sum_{k=1}^{N} m_{n_{k}} (\f^{m_{n_{k}}-1}\circ\s)\right\vert ^p |(\f\circ\s)^\prime|^p \,d\mu <\e\,.
$$
(It will suffice to work just with this single fixed value of $N$.) The corresponding values $r_{n_N}$ and $\s_{n_N}$ will satisfy
\begin{equation}
 \int_{\{|(\f\circ\s_{n_N})|>r_{n_N}\}} \left\vert  \frac{1}{N} \sum_{k=1}^{N} m_{n_{k}} (\f^{m_{n_{k}}-1} \circ\s_{n_N})
 \right\vert^p |(\f\circ\s_{n_N})^\prime|^p\,d\mu <\e\,.
 \label{epsil}
\end{equation}
For an arbitrary but fixed $z$ such that $|(\f\circ\s_{n_N})(z)| >r_{n_N}$, let us use the shorthand $x=|(\f\circ\s_{n_N})(z)|$. Then, using the triangle inequality for complex numbers, the obvious inequalities $r_{n_N}<x<1$, the elementary identity for the sum of the first $N-1$ positive integers, and \eqref{sequence}, together with the fact that the sequence $(C_k)_k$ is increasing and the obvious inequalities $m_{n_N-1}\ge m_{n_{N-1}}$ and $n_N\ge N$, it follows that
\begin{eqnarray*}
 \frac{1}{N} \left| \sum_{k=1}^{N} m_{n_{k}} (\f^{m_{n_{k}}-1} \circ\s_{n_N})(z) \right| & \ge & \frac{1}{N} \( m_{n_N} x^{m_{n_N}-1} - \sum_{k=1}^{N-1} m_{n_{k}} x^{m_{n_k}-1} \)
\\
 & \ge & \frac{1}{N} \( m_{n_N} x^{m_{n_N}-1} - \sum_{j=1}^{m_{n_{N-1}}} j x^{j-1} \)
\\
 & \ge & \frac{1}{N} \( m_{n_N} x^{m_{n_N}-1} - \sum_{j=1}^{m_{n_{N-1}}} j \)
\\
 & \ge & \frac{1}{N} \( m_{n_N} x^{m_{n_N}-1} - \frac{m_{n_{N-1}} (m_{n_{N-1}}+1)}{2} \)
\\
 & \ge & \frac{1}{N} \( m_{n_N} r_{n_N}^{m_{n_N}-1} - C_{m_{n_{N-1}}} \)
\\
 & \ge & \frac{1}{N} \( C_{m_{n_N}-1} + n_N - C_{m_{n_{N-1}}} \)
\\
 & \ge & \frac{1}{N} n_N \ge 1\,.
\end{eqnarray*}
Together with \eqref{epsil}, this yields
$$
 \int_{\{|(\f\circ\s_{n_N})|>r_{n_N}\}} |(\f\circ \s_{n_N})^\prime|^p\,d\mu <\e\,.
$$
Since this must hold for an arbitrary choice of $\e$, it contradicts our assumption \eqref{delta}. This completes the proof that \eqref{e1} holds.
\par\smallskip
\underline{\it Step 2\/}: Next, we show that the above condition \eqref{e1}, together with the weak compactness of $C_\f$, implies the following condition:
\begin{equation}
 \lim_{r\to 1} \sup_{\s\in Aut(\D)} \int_{ \{|\f\circ\s|>r\} } |(f\circ\f\circ\s)^\prime|^p\,d\mu=0\,, \quad f\in\cb\,.
 \label{e2}
\end{equation}
For any constant function the above condition is trivially fulfilled so let $f$ be an arbitrary but fixed non-constant function in $\cb$. Pick an increasing sequence $(r_n)$ convergent to $1$. Let us agree to denote by $f_r$ the dilations of $f$ defined in the usual way:
\begin{equation}
 f_r (z) := f(r z)\,, \quad 0<r<1\,.
 \label{dilat}
\end{equation}
In view of the obvious inequality:
$$
 (1-|z|^2) r_n |f^\prime (r_n z)| \le (1- |r_n z|^2) |f^\prime (r_n z)| \le \|f\|_\cb\,,
$$
the sequence $(f_{r_n})$ is a bounded in the Bloch space. Also, it  converges to $f$ uniformly on compact sets. Since the operator $C_{\f} \,\colon\,\mathcal{B}\rightarrow M(D^p_\mu)$ has the Banach-Saks property (in reality, it suffices to use the fact that it is weakly compact), there exists a subsequence of $(r_n)$, denoted in the same way by an abuse of notation, such that
$$
 \lim_{N\to\infty} \left\Vert \frac{1}{N}\sum_{k=1}^{N} (f_{r_{k}}\circ\f) - (f\circ \f) \right\Vert _{M(D^p_\mu)} = 0.
$$
Our axiom on boundedness of the point evaluations on $D^p_\mu$ implies uniform convergence of $\frac{1}{N}\sum_{k=1}^{N} (f_{r_{k}}\circ\f)$ to $f\circ \f$. We are, of course, not interested in the trivial case when the symbol $\f$ is a constant function. Since neither of the functions $f$ and $\f$ is identically constant, the same is true of the function $f\circ\f$. Thus, we may select a further subsequence, denoted again by $(r_n)$ in order not to burden the notation, so that $\frac{1}{N}\sum_{k=1}^{N} (f_{r_{k}}\circ \f)$ is not identically constant, and since $\f$ is not identically constant, this implies
$$
 \left\Vert  \frac{1}{N}\sum\limits_{k=1}^{N}f_{r_{k}}^{\prime} \right\Vert_{H^{\infty}} \neq 0 \,.
$$
Given $\e>0$, there exists a positive integer $N$ such that
$$
 \left\Vert \frac{1}{N} \sum\limits_{k=1}^{N} (f_{r_{k}}\circ\f) - (f\circ\f) \right\Vert _{M(D^p_\mu)}<\frac{\e}{2}.
$$
Moreover, by \eqref{e1} there exists $r_0\in (0,1)$ such that if $r_{0}\leq r<1$ then
$$
 \sup_{\s} \int_{ \{ |\f\circ\s| >r\} }\left\vert  (\f\circ\s)^{\prime}\right\vert\,d\mu<\frac{\e}{2 \left\Vert
 \frac{1}{N}\sum\limits_{k=1}^{N}f_{r_{k}}^{\prime}\right\Vert
 _{H^{\infty}}}\,.
$$
Hence for $r_0 \le r<1$ and for every disk automorphism $\s$ we have
\begin{align*}
 \int_{ \{|\f\circ\s|>r\} } \left\vert (f\circ\f\circ\s)^{\prime}\right\vert\,d\mu & \le
 \int_{ \{ |\f\circ\s| >r\} }\left\vert \left(
 \frac{1}{N} \sum\limits_{k=1}^{N} f_{r_{k}} \circ\f
 \circ\s\right)^{\prime} - (f\circ\f\circ\s)^{\prime}\right\vert
\,d\mu
\\
 & \qquad\qquad + \int_{ \{ |\f\circ\s| >r\} } \left\vert \left(	 \frac{1}{N}\sum\limits_{k=1}^{N}f_{r_{k}}\circ\f\circ\s\right)
 ^{\prime}\right\vert\,d\mu
\\
 & \leq\frac{\e}{2} + \int_{ \{ |\f\circ\s| >r\}  }\left\vert \left(  \frac{1}{N} \sum_{k=1}^{N}f_{r_{k}}^{\prime}\circ\f\circ\s\right)
 \left( \f\circ\s\right)^{\prime}\right\vert\,d\mu
\\
 & \leq\frac{\e}{2}+\left\Vert \frac{1}{N}\sum\limits_{k=1}^{N}f_{r_{k}}^{\prime} \right\Vert
 _{H^{\infty}}\,\int_{ \{ |\f\circ\s| >r\} }\left\vert \left(
 \f\circ\s\right)	^{\prime}\right\vert
\,d\mu\leq\e.
\end{align*}
Taking the supremum over all automorphisms $\s$, we obtain \eqref{e2}.
\par\smallskip
\underline{\it Step 3\/}: Finally, in order to see that our condition \eqref{e2} implies
$$
 \lim_{r\to 1} \sup_{\s\in Aut(\D)} \int_{ \{|\f\circ\s|>r\} } |(\f\circ\s)^\#|^p\,d\mu=0\,,
$$
which is (c), it suffices to recall Lemma~\ref{lem-ur}: there exist functions $f$ and $g$ in $\cb$ such that
$$
 |f^\prime(z)|^p + |g^\prime (z)|^p \ge \frac{C}{(1-|z|^2)^p}\,, \quad
 z\in\D\,.
$$
By applying this inequality at the point $\f(\s(z))$ instead of $z$ and then using \eqref{e2}, we see that (c) follows immediately.
\end{proof}
\par\medskip
\begin{ex}
 \label{ex-comp}
Let $1\le p<\infty$ and let $\mu$ be an arbitrary measure (not necessarily finite) that satisfies our axioms. Then for any analytic symbol $\f$ such that $\overline{\f(\D)}$ is a compact subset of $\D$, the operator  $C_\f$ is compact (equivalently, weakly compact) from $\cb$ into $M(D^p_\mu)$. Indeed, condition (c) in Theorem~\ref{thm-cpt-inv} is trivially verified. An obvious example is $\f(z)=a z+b$ with $|a|+|b|<1$, $a\neq 0$.
\end{ex}
\par\medskip
It should be made clear that not every bounded composition operator from $\cb$ into $M(D^p_\mu)$ will be compact so the above theorem describes a non-trivial situation. Here is a very simple example.
\par\medskip
\begin{ex}
 \label{ex-non-comp}
Let $d\mu=(1-|z|^2)^p\,dA(z)$, $1\le p=\a<\infty$. Then the symbol $\f(z)=(1+z)/2$ induces a bounded composition operators from $\cb$ into $M(D^p_\mu)$ which is not compact. Indeed, recall first that in this case $M(D^p_\mu)=\cb$ and that every self-map of the disk induces a bounded composition operator on $\cb$. On the other hand, our operator is not compact because the sequence $(z^n)_n$ is bounded in $\cb$ and converges to zero uniformly on compact subsets of the unit disk but evaluation at the points $z_n=\frac{n-1}{n+1}$ yields
$$
 \|C_\f (z^n)\|_\cb \ge \sup_{z\in \D} (1-|z|^2)\left| \left( \left(\frac{1+z}{2}\right)^n\right )^\prime \right|\geq 2 \left( \frac{n}{n+1}\right )^{n+1} \to \frac2{e} \neq 0\,.
$$
\par\smallskip
\end{ex}
\par\medskip
Related examples of this kind in the specific context of operators from $\cb$ into spaces of $Q_\a$ type can be found in \cite{LNST} or \cite{SZ}.
\par\medskip
\begin{cor}
 \label{cor-subsp}
Suppose that the operator $C_{\f}\,\colon\,\mathcal{B}\rightarrow M(D^p_\mu)$ is bounded. Then the following assertions are equivalent.
\begin{enumerate}
\item
$C_{\f}\,\colon\,\mathcal{B}\rightarrow M(D^p_\mu)$ is not 	 compact.
\item
There exists a subspace $X$ of $\mathcal{B}$ isomorphic to
$l^{\infty}$ such that the restriction of $C_\f$ to $X$ is an isomorphism.
\end{enumerate}
\end{cor}
\begin{proof}
It suffices to apply the fact already mentioned that $\cb$ is isomorphic to $l^{\infty}$ as a Banach space and the following result \cite[Ch.~VI, Cor.~3, p.~149]{DU}: an operator $T$ defined in $l^{\infty}$ is not weakly compact if and only if there exists a subspace $X$ of $\,l^{\infty} $ isomorphic to $l^{\infty}$ and such that $T|_X$ is an isomorphism.
\end{proof}
\par\medskip
It is natural to ask whether there is a version of Proposition~\ref{prop-bded-basic} for compact operators into $D^p_\mu$. Such a result can be easily proved by following and simplifying an easy part of the proof of Theorem~\ref{thm-cpt-inv}. It also shows that in this case examples like the last one are not possible.
\par\medskip
\begin{thm}
 \label{thm-cpt-d_p}
Let $1\leq p<\infty$. If the composition operator $C_\f$ is bounded from $\cb$ to $D^p_\mu$, it is also compact.
\end{thm}
\begin{proof}
Let $C_\f$ be bounded from $\cb$ into $D^p_\mu$. By Proposition~\ref{prop-bded-basic}, we know that
$\int_\D |\f^\#|^p\,d\mu<\infty$.
\par
Let $(r_n)_n$ be an increasing sequence with $\lim_n r_n=1$ and define $$
 \O_n=\{z\in \D\,\colon\,|\f(z)|\le r_n \}\,.
$$
Note that $\Omega_n$ is an ascending chain in the sense of inclusion whose union is $\D$. Denoting by $\chi_{\Omega_n}$ the characteristic function of $\Omega_n$, it is clear that
$|\f^\#|^p\,\chi_{\Omega_n}$ converges to $|\f^\#|^p$ pointwise and $|\f^\#|^p\,\chi_{\Omega_n}\le |\f^\#|^p$ in $\D$. By the Lebesgue dominated convergence theorem, we have
$$
 \lim_{n\to\infty} \int_{\Omega_n} |\f^\#|^p\,d\mu = \lim_{n\to\infty} \int_\D  |\f^\#|^p\chi_{\Omega_n}\,d\mu = \int_\D  |\f^\#|^p\,d\mu\,.
$$
This shows that
$$
 \lim_{n\to\infty} \int_{ \{|\varphi|>r_n\} } |\f^\#|^p\,d\mu=0 \,.
$$
Now one can just retrace the steps of the proof of the implication (c)\,$\Rightarrow$\,(a) in Theorem~\ref{thm-cpt-inv} and simplify them (without taking the supremum and working only with the identity automorphism) to see that the last  condition is sufficient for compactness of $C_\f$.
\end{proof}
\par\medskip
\section{Operators from the Bloch space into the little spaces}
 \label{sec-little-sp}
\par
The purpose of this section is to characterize the bounded and compact operators from $\cb$ into the small spaces $M_0(D^p_\mu)$. The proof of this characterization will require some ``obvious'' properties such as separability of $M_0(D^p_\mu)$ which are well known to hold in the classical ``little spaces'' like VMOA, $\cb_0$, and $Q_{\a,0}$. For example, this property is fulfilled whenever the polynomials are dense in the space. However, in our general context separability has to be checked and it turns out that a complete and rigorous proof of this fact is somewhat involved.
\par
In what follows, we shall typically (but not exclusively) consider positive measures $\mu$ of the form  $d\mu(z)=h(|z|)\,dA(z)$, where $h\in [0,1)\to [0,\infty)$ is an integrable function. Moreover, we shall assume that there exist positive constants $\a$ and $C$ such that
\begin{equation}
 h(|\s(z)|)\le C\,h(|z|)|\,\s^\prime(z)|^{\a}
 \label{cond-h}
\end{equation}
for all $z\in \D$ and all $\s\in Aut(\D)$. Then the induced measure $\mu$ is finite. We will refer to such $\mu$ as the \textit{radial measure induced by\/} $h$. We remind the reader that the definitions of all classical conformally invariant spaces involve measures of this type.
\par
Let us agree to write
\begin{equation}
 h_\s:= (h\circ|\s|) |\s^\prime|^{2-p}.
 \label{g_s}
\end{equation}
Using the standard change of variable: $z=\s(w)$, $dA(z)= |\s^\prime (w)|^2\,dA(w)$, it is easy to verify the identity
\begin{equation}
 \int_\D|(f\circ\s^{-1})^\prime(z)|^pd\mu(z)=\int_\D |f^\prime(w)|^ph_\s(w)\,dA(w)
 \label{change_of_variable}
\end{equation}
for every function $f$ in $M(D^p_\mu)$.
\par
The first natural question is: when does $M_0(D^p_\mu)$ contain the polynomials?
\par\smallskip
\begin{prop}
 \label{prop-polyn-small}
Let $\mu$ be a radial measure induced by an integrable,  non-negative, radial function $h$. Then the following statements are equivalent:
\begin{itemize}
\item[(a)]
The identity function, given by $f(z)=z$, belongs to $M_0(D^p_\mu)$.
\item[(b)]
All polynomials belong to $M_0(D^p_\mu)$.
\item[(c)]
The following two conditions hold simultaneously:
\begin{equation}
 \sup_\s \int_\D h_\s(w)dA(w)<\infty\,, \quad \lim_{\s\in Aut(\D),\,|\s(0)|\to 1} \int_\D h_\s(w)dA(w) =0\,.
 \label{polyn-small}
\end{equation}
\end{itemize}
\end{prop}
\begin{proof}
Formula \eqref{change_of_variable} readily implies that the identity function, given by $f(z)=z$, belongs to $M_0(D^p_\mu)$ if and only if \eqref{polyn-small} holds.
\par
Trivially, if all polynomials belong to $M_0(D^p_\mu)$ then so does $f(z)=z$. If the identity is in $M_0(D^p_\mu)$ then \eqref{polyn-small} holds, and choosing $f(z)=z^n$ we get
$$
 \sup_\s \int_\D|(f\circ \s^{-1})^\prime|^pd\mu=\sup_\s \int_\D
 |f^\prime|^ph_\s dA\leq n^p \sup_\s \int_\D h_\s dA<\infty
$$
and
\begin{eqnarray*}
 \lim_{|\s(0)|\to 1}\int_\D|(f\circ \s^{-1})^\prime|^pd\mu &=&
 \lim_{|\s(0)|\to 1}\int_\D|f^\prime(w)|^ph_\s(w)dA(w)\\
 &\leq &n^p\lim_{|\s(0)|\to 1}
 \int_\D h_\s(w)dA(w)=0.
\end{eqnarray*}
It is easy to see that $|\s^{-1}(0)|\to 1$ if and only if $|\s(0)|\to 1$. Recalling also the obvious fact that $\{\s^{-1}\,\colon\,\s\in Aut(\D)\}=Aut(\D)$, the statement follows.
\end{proof}
\par\medskip
\begin{thm} \label{thm-dens-polyn}
Let $\mu$ be a radial measure induced by an integrable, non-negative, radial function $h$ that satisfies \eqref{cond-h}; suppose also that the identity function belongs to $M_0(D^p_\mu)$. Let the dilations $f_r$ be defined as in \eqref{dilat}. Then the following statements are equivalent:
\begin{itemize}
\item[(a)] The function $f$ belongs to $ M_0(D^p_\mu)$;
\item[(b)] $\lim_{r\to1} \Vert f-f_r\Vert_{M(D^p_\mu)}=0.$
\item[(c)] $f$ belongs to the closure of the polynomials in $M(D^p_\mu)$.
\end{itemize}
\end{thm}
\begin{proof}
Our proof will consist of proving the chain of implications (a)\,$\Rightarrow$\,(b)\,$\Rightarrow$\, (c)\,$\Rightarrow$\,(a).
\par\medskip
\fbox{(a)\,$\Rightarrow$\,(b).} Let $f\in M_0(D^p_\mu)$. The key points is that, by assumption, the identity belongs to $M(D^p_\mu)$. Also, all  dilations $f_r$ have a continuous derivative in the closed disk. However, we will need a uniform bound on their norms in terms of $f$. Using the Poisson's kernel, we can rewrite the function $f_r$ as
$$
 f_r(z)=\frac{1}{2\pi}\int_{\T}f(z\xi)\frac{1-r^2}{|1-r \overline\xi|^2}|d\xi|.
$$
Thus
$$
 f^\prime_r(z)=\frac{1}{2\pi}\int_{\T}\xi f^\prime(z\xi)\frac{1-r^2}{|1-r\overline\xi|^2}|d\xi|.
$$
Fix $\s\in Aut(\D)$. Given $\la\in \T$, let us define $\s_\la (z) :=\s(\overline{\la} z)$, $z\in \D$. Then, applying the equality (\ref{change_of_variable}) and Minkowski's
inequality,
\begin{eqnarray*}
 \left(\int_\D |(f_r\circ\s^{-1})^\prime|^pd\mu\right)^{1/p} &=&\left(\int_\D |f_r^\prime|^p h_\s dA\right)^{1/p}
\\
 &=&\frac{1}{2\pi}\left(\int_\D \left|\int_{\T}\xi f^\prime(z\xi)\frac{1-r^2}{|1-r\overline\xi|^2}|d\xi|\right|^p h_\s(z)dA(z)\right)^{1/p}
\\
 &\leq & \frac{1}{2\pi}\int_{\T} \left(\int_\D \left| \xi f^\prime(z\xi)\right|^p h_\s(z)dA(z) \right)^{1/p} \frac{1-r^2}{|1-r\overline\xi|^2}|d\xi|
\\
 &= & \frac{1}{2\pi}\int_{\T} \left(\int_\D \left|  f^\prime(w)\right|^p h_{\s_\xi}(w)dA(w)\right)^{1/p}\frac{1-r^2}{|1-r\overline\xi|^2}|d\xi|
\\
 &\leq & \frac{1}{2\pi}\int_{\T} \sup_{\la\in \T}\left(\int_\D \left|  f^\prime(w)\right|^p h_{\s_\la}(w)dA(w)\right)^{1/p}\frac{1-r^2}{|1-r\overline\xi|^2}|d\xi|
 \\
 &=&\sup_{\la\in \T}\left(\int_\D \left| f^\prime(w)\right|^p h_{\s_\la}(w)dA(w)\right)^{1/p} \frac{1}{2\pi}\int_{\T} \frac{1-r^2}{|1-r\overline\xi|^2}|d\xi|
\\
 &=&\sup_{\la\in \T}\left(\int_\D \left|  f^\prime(w)\right|^p h_{\s_\la}(w)dA(w)\right)^{1/p}\,.
\end{eqnarray*}
That is, for all $\s\in Aut(\D)$ and all $r\in (0,1)$ we have that
\begin{equation}
 \label{Minkowski}
 \int_\D |(f_r\circ\s^{-1})^\prime|^pd\mu\leq \sup_{\la\in\T}\int_\D |(f\circ\s_\la^{-1})^\prime|^pd\mu.
\end{equation}
It now follows that $\Vert f_r\Vert_{M(D^p_\mu)}\leq \Vert f\Vert_{M(D^p_\mu)}$ since $f_r(0)=f(0)$.
\par
In the  special case when $\s$ is chosen to be the identity, a close inspection of the above long chain of inequalities shows that
\begin{equation}
 \int_\D |f_r^\prime|^pd\mu\leq \int_\D |f^\prime|^pd\mu.
 \label{Minkowski2}
\end{equation}
Using the description  \eqref{autom} of the group of disk automorphisms, it is easy to see that for every function $f\in M_0(D^p_\mu)$ we have
\begin{equation}
 \lim_{|\s(0)|\to1}\left(\sup_{\la\in \T}\int_\D|(f\circ\s_\la^{-1})^\prime|^pd\mu\right)=0.
 \label{I^p_0_strong}
\end{equation}
\par
Assume the contrary of what we want to prove: $\lim_{r\to1} \Vert f-f_r\Vert_{M(D^p_\mu)}\neq 0$. Then there exist a constant $\delta>0$, a sequence of positive numbers $r_n\nearrow 1$, and a sequence of automorphisms of the unit disk $(\s_n)_n$ such that
\begin{equation}
 \left(\int_\D|(f\circ\s_n^{-1})^\prime-(f_{r_n}\circ\s_n^{-1})^\prime|^p \,d\mu\right)^{1/p}\ge \delta
  \label{ge-delta}
\end{equation}
for all $n$.
\par
After passing to a subsequence, we may assume that the sequence $(\s_n)_n$ converges uniformly on compact subsets of the unit disk. By a corollary to Hurwitz's theorem, it converges either to a constant $\lambda\in \T$ or to an automorphism $\s$. We analyze the two cases separately.
\par\smallskip
\underline{\it Case 1\/}. Suppose that $\s_n(0)\to\la$ with $|\la|=1$; then $|\s_n(0)|\to 1$. By (\ref{I^p_0_strong}), we can find $n_0\in\N$ such that
$$
 \sup_{\la\in \T} \left(\int_\D |(f\circ\s_{n,\la}^{-1})^\prime|^p\, d\mu \right)^{1/p} < \delta/4
$$
for all $n\ge n_0$, where $\s_{n,\la} (z)=\s_n(\overline{\la} z)$ as before. From here we deduce by \eqref{Minkowski} that
\begin{multline*}
 \left(\int_\D|(f\circ\s_n^{-1})^\prime-(f_{r_n}\circ\s_n^{-1})^\prime|^p
 \,d\mu\right)^{1/p}
\\
 \le \left(\int_\D|(f_{r_n}\circ\s_n^{-1})^\prime|^pd\mu\right)^{1/p} + \left(\int _\D|(f\circ\s_n^{-1})^\prime|^pd\mu\right)^{1/p}
\\
 \le 2 \sup_{\la\in \T}\left(\int_\D |(f\circ\s_{n,\la}^{-1})^\prime|^p \,d\mu\right)^{1/p}< \delta/2\,,
\end{multline*}
which is in contradiction with \eqref{ge-delta}.
\par\smallskip
\underline{\it Case 2\/}. Suppose $\s_n\to\s$ uniformly on compact sets, $\s\in Aut(\D)$. Then there exists $r\in (0,1)$ such that $|\s_n^{-1}(0)|\leq r$ for all $n\in\N$. Thus,
$$
 \frac{1-r}{1+r}\leq|(\s_n^{-1})^\prime(z)|\leq \frac{1+r}{1-r}
$$
holds for all $z\in \D$. Taking $\a$ to be the same constant as in \eqref{cond-h}, denote by $D$ the following finite positive constant:
$$
D=\max\left\{\left(\frac{1-r}{1+r}\right)^{2-p+\a},  \left(\frac{1+r}{1-r}\right)^{2-p+\a}\right\}\,.
$$
Extend $h$ to $\D$ radially by defining $h(z)=h(|z|)$. By formula (\ref{change_of_variable}) and our hypotheses on $h$, we have
\begin{multline*}
 \int_\D|(f\circ\s_n^{-1})^\prime - (f_{r_n}\circ\s_n^{-1})^\prime|^p \,d\mu = \int_\D|f^\prime-f_{r_n}^\prime|^p\,h_{\s_n}\,d A
\\
 \le C \int_\D |f^\prime(w)-f_{r_n}^\prime(w)|^ph(|w|)|\s_n'(w)|^{2-p+\a} \,d A (w)
\\
 \le C D \int _\D|f^\prime(w)-f_{r_n}^\prime(w)|^p\,h(|w|)\,d A (w)
\\
 \le C D p \int _\D(|f^\prime(w)|^p+|f_{r_n}^\prime(w)|^p)\,h(|w|)\,d A (w).
\end{multline*}
Notice that $|f^\prime|^ph$ belongs to $L^1(\D,dA)$ because $f\in M(D^p_\mu)$ and $|f_{r_n}^\prime|^ph$ also belongs to $L^1 (\D,dA)$ by (\ref{Minkowski2}).
Since $|f_{r_n}^\prime|$ converges pointwise to $|f^\prime|$, Fatou's lemma and again (\ref{Minkowski2}) together imply that
\begin{equation*}
\int_\D|f^\prime|^pd\mu \leq \liminf_n\int_\D|f_{r_n}^\prime|^pd\mu\leq \limsup_n\int_\D|f_{r_n}^\prime|^pd\mu =\int_\D|f^\prime|^pd\mu.
\end{equation*}
Thus,
\begin{equation}
\label{Minkowski+Fatou}
\lim_n\int_\D|f_{r_n}^\prime|^pd\mu=\int_\D|f^\prime|^pd\mu.
\end{equation}
In summary, we know the following:
\begin{itemize}
\item
 $|f_{r_n}^\prime-f^\prime|^p\,h_{\s_n} \leq C D p \,(|f_{r_n}^\prime|^p+|f^\prime|^p) h$ holds throughout $\D$;
\item
 $|f_{r_n}^\prime-f^\prime|^p\,h_{\s_n}\to 0$ pointwise in $\D$ as $n\to\infty$;
\item
 $(|f_{r_n}^\prime|^p+|f^\prime|^p)h\to 2|f^\prime|^p\,h$ pointwise in $\D$;
\item
 $\int_\D (|f_{r_n}^\prime|^p+|f^\prime|^p)\,h\,dA \to \int_\D 2|f^\prime|^p\,h\,dA$ by \eqref{Minkowski+Fatou}.
\end{itemize}
Thus, we are allowed to apply the well-known generalization of the Lebesgue dominated convergence theorem usually called Pratt's lemma \cite[Ch.~11, Sect.~5, p.~232, Prop.~18]{Ro}, obtaining
$$
 \int_\D|(f\circ\s_n^{-1})^\prime-(f_{r_n} \circ \s_n^{-1})^\prime|^p\,d\mu = \int_\D |f^\prime-f_{r_n}^\prime|^p \,h_{\s_n}\,dA \to 0\,, \quad n\to\infty\,,
$$
which again yields a contradiction. This concludes our proof that (a)\,$\Rightarrow$\,(b).
\par\medskip
\fbox{(b)\,$\Rightarrow$\,(c).} Fix $r\in (0,1)$. Since $f_r$ is analytic in a larger disk centered at the origin, there exists a sequence of polynomials $(p_n)_n$ such that $p_n(0)=f(0)$ and
$$
M_n:=\sup\{|f_r^\prime(z)-p_n'(z)|\,\colon\,\, z\in \D\}\to 0 \quad \textrm{as} \quad n\to \infty.
$$
Thus,
\begin{eqnarray*}
\sup_\s \int_\D |(f_r\circ\s^{-1})^\prime - (p_n\circ\s^{-1})^\prime|^pd\mu  &=&\sup_\s \int_\D |f_r^\prime - p_n'|^ph_\s\,dA\\
&\leq& M_n^p\sup_\s \int_\D h_\s dA=M_n^p\,\Vert z\Vert_{M(D^p_\mu)}^p\,.
\end{eqnarray*}
Therefore, $\lim_n\Vert f_r-p_n\Vert_{M(D^p_\mu)}=0$ and $f_r$ belongs to the closure of the polynomials in $M(D^p_\mu)$ for all $r>1.$ Since $\lim_{r\to1} \Vert f-f_r\Vert_{M(D^p_\mu)}=0$, the function $f$ also belongs to the closure of the polynomials.
\par\medskip
\fbox{(c)\,$\Rightarrow$\,(a).} Since the polynomials belong to $M_0(D^p_\mu)$ and $M_0(D^p_\mu)$ is a closed subspace of $M(D^p_\mu)$, it is clear that $f\in M_0(D^p_\mu)$.
\end{proof}
\par\medskip
We would like to remark that here our inspiration for the above result comes from \cite{WX}. Since we are considering compositions with all automorphisms, there are some technical difficulties involved. The radial character of the measure $\mu$ does not seem to guarantee all rotation-invariant properties to hold so it appears necessary to consider the supremum over all compositions with rotations as we have done above or to use some similar argument in order to complete the  proof.
\par\medskip
\begin{cor}
 \label{cor-dens-polyn}
Let $\mu$ a radial measure induced by a positive, radial, integrable function $h$ that satisfies \eqref{cond-h}. If the identity function is in $M_0(D^p_\mu)$ then the polynomials are dense in $M_0(D^p_\mu)$. In particular, the Banach space $M_0(D^p_\mu)$ is separable.
\end{cor}
\par\smallskip
We end the paper with a characterization of bounded and compact composition operators from the Bloch space into the small spaces.
\par\smallskip
\begin{thm} \label{thm-small}
Let $1\leq p<\infty$. Assume that the Banach space $M_0(D^p_\mu)$ is separable. Let  $\f$ be an analytic self-map of the unit disk.
Then the following statements are equivalent:
\begin{itemize}
 \item[(a)]
$C_\f\,\colon\,\cb \to M_0(D^p_\mu)$ is a bounded operator.
 \item[(b)]
$C_\f\,\colon\,\cb \to M_0(D^p_\mu)$ is a compact operator.
 \item[(c)]
Both conditions
$$
 \lim_{|\s(0)|\to 1}\int_\D|(\f\circ\s)^\#|^pd\mu=0 \quad and \quad \sup_{\s\in Aut(\D)}\int_\D| (\f\circ\s)^\#|^p\,d\mu<\infty
$$
are fulfilled.
\end{itemize}
\end{thm}
\begin{proof}
We will first show that the conditions (a) and (b) are equivalent and then also that (a) is equivalent to (c).
\par\medskip
\fbox{(b)\,$\Rightarrow$\,(a)} is trivial.
\par\medskip
\fbox{(a)\,$\Rightarrow$\,(b).} As was already remarked, $\cb$ is isomorphic to $l^\infty$; see, \textit{e.g.}, \cite{Lu}. Therefore every bounded operator from $\cb$ into a separable Banach space is weakly compact (see \cite[Ch.~VI, Cor.~3, p.~149]{DU}, for example). In particular, every bounded operator from $\cb$ into $M_0(D^p_\mu)$ is weakly compact. Since $ M_0(D^p_\mu)$ is a subspace of $M(D^p_\mu)$, it follows that $C_\f\,\colon\, \cb \to M(D^p_\mu)$ is also a weakly compact operator. By Theorem \ref{thm-cpt-inv}, it is compact.
\par\medskip
\fbox{(a)\,$\Rightarrow$\,(c).} By Lemma~\ref{lem-ur}, there exist $f$, $g\in \cb$ and a positive constant $C$ such that
\begin{equation}
 |f^\prime(z)|^p+|g^\prime(z)|^p\geq \frac{C}{(1-|z|^2)^p},\quad \textrm{ for all } z\in \D.
\end{equation}
Thus, given $\s\in Aut(\D)$, we have that
\begin{eqnarray*}
 |(f\circ\f\circ\s)^\prime|^p+|(g\circ\f\circ\s)^\prime|^p&=&
 \left(|f^\prime\circ\f\circ\s|^p+|g^\prime\circ\f\circ\s|^p\right)|(\f\circ \s)^\prime|^p
\\
 &\geq& \frac{C}{(1-|\f\circ\s)|^2)^p}|(\f\circ\s)^\prime|^p\\
 &=&C|(\f\circ\s)^\#|^p\,.
\end{eqnarray*}
Since $C_\f$ maps the Bloch space into $M_0(D^p_\mu)$, the functions $f\circ\f$ and $g\circ\f$ belong to $M_0(D^p_\mu)$ so from the above inequalities we deduce that
$$
 \lim_{|\s(0)|\to 1}\int_\D|(\f\circ\s)^\#|^pd\mu=0\,.
$$
By (a), the operator $C_\f\,\colon\, \cb \to M(D^p_\mu)$ is bounded and, by Proposition~\ref{prop-bded-inv}, we have
$$
 \sup_\s \int_\D|(\f\circ\s)^\#|^pd\mu <\infty\,.
$$
\par
\fbox{(c)\,$\Rightarrow$\,(a)}. Applying again Proposition~\ref{prop-bded-inv}, we conclude that the composition operator $C_\f$ is bounded from the Bloch space into $M(D^p_\mu)$. It is only left to prove that the range of $C_\f$ is contained in the little space $M_0(D^p_\mu)$. To this end, suppose $f\in \cb$. Then
\begin{eqnarray*}
 |(f\circ\f\circ\s)^\prime|^p=|f^\prime\circ\f\circ\s|^p |(\f\circ\s)^\prime|^p & \le & \frac{\Vert f\Vert_{\cb}^p}{(1-|\f\circ\s|^2)^p}|(\f\circ\s)^\prime|^p
\\
 &=& \Vert f\Vert_{\cb}^p|(\f\circ\s)^\#|^p.
\end{eqnarray*}
By our assumption (c), it follows that
$$
\lim_{|\s(0)|\to 1}\int_\D|(f\circ\f\circ\s)^\prime|^pd\mu=0\,,
$$
hence $f\circ\f\in M_0(D^p_\mu)$. This completes the proof.
\end{proof}

\bigskip

\end{document}